\documentclass{amsart}
\usepackage{amsrefs}
\usepackage{amsthm}
\usepackage{graphicx,color}
\usepackage[english]{babel}  
\usepackage[latin1]{inputenc}
\usepackage[mathscr]{eucal}
\usepackage{amssymb}
\usepackage[dvipsnames]{xcolor} 
\begin{document}

\title[]
{A Berestycki-Lions type result and applications }

\address{Claudianor O. Alves \newline Departamento de Matem\'atica,
	Universidade Federal de Cam\-pi\-na Grande,
	58429-010, Campina Grande - PB, Brazil}

\email{}

\author[C. O. Alves, R.C. Duarte and M.A.S. Souto]{Claudianor O. Alves, Ronaldo C. Duarte and Marco A. S. Souto}



\thanks{C.O. Alves was partially supported by CNPq/Brazil Proc. 304036/2013-7, R.C. Duarte was partially supported by CAPES  and M. A.S. Souto was partially
	supported by  CNPq/Brazil Proc. 305384/2014-7}

\address{Ronaldo C. Duarte \newline Departamento de Matem\'atica,
         Universidade Federal de Cam\-pi\-na Grande,
         58429-010, Campina Grande - PB, Brazil}

\email{ronaldocesarduarte@gmail.com}

\address{Marco A. S. Souto \newline Departamento de Matem\'atica,
	Universidade Federal de Cam\-pi\-na Grande,
	58429-010, Campina Grande - PB, Brazil}

\email{marco.souto.pb@gmail.com}


\begin{abstract}
In this paper we show an abstract theorem involving the existence of critical points for a functional $I$, which permit us to prove the existence of solutions for a large class of  Berestycki-Lions type problems.  In the proof of the abstract result we apply the deformation lemma on a special set associated with $I$, which we call of Pohozaev set.
\end{abstract}

\maketitle

\newtheorem{theorem}{Theorem}[section]
\newtheorem{lemma}[theorem]{Lemma}
\newtheorem{proposition}[theorem]{Proposition}
\newtheorem{corollary}[theorem]{Corollary}
\newtheorem{remark}[theorem]{Remark}
\newtheorem{definition}[theorem]{Definition}
\newtheorem{claim}[theorem]{Claim}
\renewcommand{\theequation}{\thesection.\arabic{equation}}

\section{Introduction}

At the last years a lot of authors have dedicated a special attention for existence of solution for elliptic problems of the type  
\begin{equation}\label{BL}
- \Delta u   = g(u), \quad \mbox{in} \quad \mathbb{R}^N, 
\end{equation}
where $N \geq 2$, $\Delta$ denotes the Laplacian operator and $g$ is a continuous function verifying some conditions.  

The main motivation comes from of the seminal paper due to Berestycki and Lions \cite{berest}, which has considered the existence of solution for  (\ref{BL}) by assuming that $N \geq 3$ and the following conditions on $g$:
$$
- \infty < \liminf_{s \to 0^+}\frac{g(s)}{s} \leq \limsup_{s \to 0^+}\frac{g(s)}{s}\leq -m<0, 
$$
$$
\limsup_{s \to 0^+}\frac{g(s)}{s^{2^{*}-1}}\leq 0, 
$$
$$
\mbox{there is} \quad \xi>0 \, \, \mbox{such that} \,\, G(\xi)>0, 
$$
where $G(s)=\int_{0}^{s}g(t)\,dt$.

In \cite{BGK}, Berestycki, Gallouet and Kavian have studied the case where $N=2$ and the nonlinearity $g$ possesses an exponential growth of the type
$$
\limsup_{s \to 0^+}\frac{g(s)}{e^{\beta s^2}}=0, \quad \forall \beta >0.
$$

In the two  above mentioned papers, the authors have used the variational method to prove the existence of solution for (\ref{BL}). The main idea is to solve the minimization problems 
$$
\min \left\{\frac{1}{2}\int_{\mathbb{R}^N}|\nabla u|^{2} \,dx \,:\, \int_{\mathbb{R}^N}G(u)\,dx=1  \right\}  
$$
and 
$$
\min \left\{\frac{1}{2}\int_{\mathbb{R}^N}|\nabla u|^{2}\,dx \,:\, \int_{\mathbb{R}^N}G(u)\,dx=0  \right\}  
$$
for $N \geq 3$ and $N=2$ respectively. After that, the authors showed that the minimizer functions of the above problems are in fact ground state solutions of (\ref{BL}). By a ground state solution, we mean a solution $u \in H^{1}(\mathbb{R}^N) \setminus \{0\}$ that satisfies
$$
E(u) \leq E(v) \quad \mbox{for all nontrival solution} \  v  \ \text{of} \ (\ref{BL}),
$$ 
where $E:H^{1}(\mathbb{R}^N) \to \mathbb{R}$ is the energy functional associated to (\ref{BL}) given by
$$
E(u)=\frac{1}{2}\int_{\mathbb{R}^N}|\nabla u|^{2}\,dx - \int_{\mathbb{R}^N}G(u)\,dx.
$$

After, Jeanjean and Tanaka in \cite{JJTan} showed that the mountain pass level of  $E$ is a critical level and it is 
indeed the lowest critical level.

A version of the problem (\ref{BL}) for the critical case have been  made in Alves, Souto and Montenegro \cite{AlvesSoutoMontenegro} for $N \geq 3$ and $N=2$, see also Zhang and Zhou \cite{ZZ} for $N=3$. The reader can found in Alves, Figueiredo and Siciliano \cite{AGS}, Chang and Wang \cite{ChangWang} and Zhang, do \'O and Squassina \cite{ZOS} the same type of results involving the fractional Laplacian operator, more precisely, for a problem like 
\begin{equation}\label{fracionario}
(- \Delta)^{\alpha}u = g(u), \quad \mbox{in} \quad \mathbb{R}^N, 
\end{equation}
with $\alpha \in (0,1)$ and $N \geq 1$. 

We would like point out that the method used in the above papers works well because $g$ does not depend on $x$, $-\Delta$ and $(-\Delta)^{\alpha}$ are homogeneous operators and there is a Pohozaev identity associated  with (\ref{BL}) and (\ref{fracionario}) . When one of these facts is not verified it is necessary to change the arguments. In \cite{AW}, Pomponio and Watanabe have studied the existence of solution for (\ref{BL}) changing $-\Delta$ by $-\Delta_p - \Delta_q$, that is, the following problem has been considered 
\begin{equation}\label{PW}
- \Delta_p u- \Delta_q u   = g(u), \quad \mbox{in} \quad \mathbb{R}^N. 
\end{equation}
In this case, there is a loss of homogeneity in the operator,  then the arguments used in the previous paper does not work well. To overcome this difficulty, the authors have used a result found in Jeanjean \cite[Theorem 1.1]{jeanjean} to solve the problem. However, the fact that above problem has a Pohozaev identity is crucial in their approach.  In \cite{AP}, Azzollini and Pomponio have considered the existence of solution for the following class of problem 
\begin{equation}\label{AP}
- \Delta u + V(x)u   = g(u), \quad \mbox{in} \quad \mathbb{R}^N. 
\end{equation}
By supposing some geometry conditions on $V$, the authors also used  \cite[Theorem 1.1]{jeanjean} as well as the fact that there is a  Pohozaev identity associated with the problem. 

The read is invited to see that in the papers \cite{AW} and \cite{AP} a Pohozaev identity is a key point to prove that a sequence of approximate solutions for (\ref{PW}) and (\ref{AP}) are  bounded.

Motivated by above papers, we were led for the following problem: If $g$ is a discontinuous function how to get a solution for problems (\ref{BL}), (\ref{fracionario}) or (\ref{PW}) ?  The main difficulty is related to the fact that the classical variational methods to $C^1$ functional cannot be applied. Moreover, in this situation, there is no  Pohozaev identity associated with the problem. A second problem that we are interesting is  a version of (\ref{PW}) involving fractional Laplacian problem, more precisely, 
\begin{equation}\label{2fracionario}
(- \Delta)^{\alpha} u + (- \Delta)^{\beta} u   = g(u), \quad \mbox{in} \quad \mathbb{R}^N,
\end{equation}
where $\alpha, \beta \in (0,1)$, $N > 2\max\{\alpha, \beta\}$ and $g$ be a continuous function. This problem becomes interesting, because the authors do not know a Pohozaev identity associated with it. Finally, another problem that we are interesting is the following Anisotropic problem 
\begin{equation} \label{ANP0}
-	\sum_{i=1}^{n}\frac{\partial}{\partial x_i}\left(\left|\frac{\partial u}{\partial x_i} \right|^{p_i-2}\frac{\partial u}{\partial x_i}\right)=g(u), \quad \mbox{in} \quad \mathbb{R}^N,
\end{equation}
where $1< p_1 < ...< p_n <N$ and $g$ be a continuous function. Here, as in last problem the authors do not know a Pohozaev identity associated with (\ref{ANP0}).  From these above commentaries a new approach must be developed to solved them. Having theses problems in mind, in the present paper we show that there is an abstract result behind of the famous result due to Berestycki and Lions \cite{berest}, which can be used to solve a lot of problems where the existence of a Pohozaev identity is not clear. An advantage of our main result is related to the fact that it can be applied for a lot of problems where the nonlinearity $g$ is continuous, and also, for some problems where the nonlinearity $g$ is discontinuous.

Before stating our abstract theorem we need to fix some notations. In the sequel $(X, ||\,\,\,||)$ is a reflexive Banach space and $\psi_{1},..., \psi_{n}, \Phi: X \rightarrow \mathbb{R}$ are continuous functionals verifying: \\

\noindent There is an application $\ast: [0, \infty) \times X \rightarrow X$ and $\lambda_{1},...,\lambda_{n}, \lambda_{\Phi} \in \mathbb{R}$ such that 

		\noindent $(X_1)$ \,\, $\psi_{i}(\ast(t,u)) = t^{\lambda_{i}}\psi_{i}(u);$ \\
		\noindent $(X_2)$ \,\, $\Phi(\ast(t,u)) = t^{\lambda_{\Phi}}\Phi(u);$\\
		\noindent $(X_3)$ \,\, $0 < \max\left\{\lambda_{1}, ..., \lambda_{p}\right\}< \lambda_{\Phi};$\\
		\noindent $(X_4)$ \,\, $\ast(0,u)=0$, \quad $\forall u \in X;$\\
		\noindent $(X_5)$ For each $u \in X$ fixed, the application $t \longmapsto \ast(t,u)$ is continuous.\\

\noindent  There are subsets $X^+, X^r \subset X$ that are weak closed and a function $Q:X^+ \rightarrow X^r$ verifying: \\

\noindent $(X_6)$ \,\,  $\psi_{i}(Q(u)) \leq \psi_{i}(u)$, $ \forall u \in X^+$, \, $\forall i \in \{1,...,n\}$; \\
\noindent $(X_7)$ \,\,  $\Phi(Q(u))\geq \Phi(u)$, $ \forall u \in X^+$;\\
\noindent $(X_8)$\,\,  If $u\in X^r$, then $\ast(t,u) \in X^r$ for all $t \geq 0$.\\

Before writing our next assumptions, we would like to fix the following notations
$$
u_{t}:=\ast(t,u), \quad \forall t \geq 0 \quad \mbox{and} \quad u \in X,
$$
$$
J(u):= \sum_{i=1}^{n} \psi_{i}(u), \quad \forall u \in X
$$
and
$$
I(u)=J(u)- \Phi(u), \quad \forall u \in X.
$$
Moreover of the above conditions, we also assume the following: \\

\noindent  $(F_{1})$ $\Phi(0)=0$ and there is $u \in X$ such that $\Phi(u)>0.$ \\
\noindent  $(F_{2})$ $\psi_{i}(u) \geq 0$ for all $i \in \{1,...,n\}$ and $u \in X$. Moreover, $J(u) = 0 \Leftrightarrow u=0$.\\
\noindent  $(F_{3})$ There exists $r>0$ such that if $0<||u||< r$, then  
	$$
	\sum_{i=i}^{n}\lambda_{i}\psi_{i}(u)>\lambda_{\Phi}\Phi(u).
	$$
\noindent $(F_{4})$ For any sequence $(u_{k})$ satisfying $\Phi(u_{k})\geq 0$ and $J(u_{k}) \rightarrow 0$, we have 
	$ ||u_{k}|| \rightarrow 0.$ Moreover, if  $(J(u_{k}))$ is bounded, then $(u_{k})$ is also bounded. \\

\noindent $(F_{5})$ If $(u_{k}) \subset X^r$ is weakly convergent for $u$ in $X$, then 
	$$
	\limsup_{k \rightarrow \infty}\Phi(u_{k}) \leq \Phi(u). 
	$$
\noindent $(F_{6})$ If $(u_{k})$ is weakly convergent for $u$ in $X$, then   
	$$ 
	\psi_{i}(u) \leq \liminf_{k \rightarrow \infty}\psi_{i}(u_{k}), \quad \forall i \in \{1,2,....,n\}.
	$$

In throughout this article, we denote by $\mathcal{P}$ and $\mathcal{P}^+$ the sets 
$$
	\mathcal{P}=\left\{u \in X\setminus \left\{0\right\}: \lambda_{1}\psi_{1}(u)+...+\lambda_{n}\psi_{n}(u)=\lambda_{\Phi}\Phi(u)\right\}
$$
and
$$
\mathcal{P}^+=\left\{u \in X^+  \setminus \left\{0\right\}: \lambda_{1}\psi_{1}(u)+...+\lambda_{n}\psi_{n}(u)=\lambda_{\Phi}\Phi(u)\right\}.
$$

The set $\mathcal{P}$ will called of {\it Pohozaev set} and associated with it we have the operator 
$$
	K(u)=\lambda_{1}\psi_{1}(u)+...+\lambda_{n}\psi_{n}(u) - \lambda_{\Phi} \Phi(u),
$$
which will call of {\it Pohozaev operator}. Note that $K^{-1}(\{0\})=\mathcal{P}\cup \left\{0\right\}$.

\vspace{0.5 cm}

Now, we are ready to state our main result.

\begin{theorem}\label{T1}
	Let $X$, $\Phi$, $\psi_{1}$,..., $\psi_{n}$ satisfying $(X_{1})-(X_8)$ and $(F_1)-(F_6)$. If 
	$$
	\inf_{w \in \mathcal{P}}I(w)=\inf_{w \in \mathcal{P}^+}I(w),
	$$ 
	then there is $u \in \mathcal{P}$ such that $I(u)= \displaystyle \inf_{w \in \mathcal{P}}I(w)>0$. If $I$ is locally Lipschitz, then $u$ is a critical point of $I$ in $X$, that is, 
	$ 0 \in \partial I (u). $
\end{theorem}

The plan of the paper is as follows: In Section 2 we have proved some preliminary results that  will be used in Section 3 to show the Theorem \ref{T1}. In Section 4 we study the existence of solution for a large class of problem, which includes the problem
$$
(- \Delta)^{\alpha} u + (- \Delta)^{\beta} u   = g(u), \quad \mbox{in} \quad \mathbb{R}^N,   \eqno{(P_1)}
$$
where $\alpha, \beta \in (0,1)$, $(- \Delta)^{\alpha}$ and $(- \Delta)^{\beta}$  denote the fractional Laplacian of order $\alpha$ and $\beta$ respectively, $N > 2 \max\{\alpha, \beta\} $ and $g(s)=f(s)-s$ is a continuous function. In Section 5, we consider the existence of solution for an Anisotropic  problem like 
$$
-\sum_{i=1}^{n}\frac{\partial}{\partial x_i}\left(\left|\frac{\partial u}{\partial x_i} \right|^{p_i-2}\frac{\partial u}{\partial x_i}\right)=g(u), \quad \mbox{in} \quad \mathbb{R}^N, \eqno({P_2})
$$
where $1<p_1<...<p_{n}<N$ and $g=f(s)-|s|^{p_0-2}s$ is a continuous function. Finally, in Section 6, we establish the existence of solution for a class of discontinuous problem of the type  
$$
-\Delta u(x)  \in \partial G(u(x)), \quad \mbox{a.e. in} \quad \mathbb{R}^N, \eqno{(P_3)}
$$
where $N \geq 1$, $G$ is the primitive of a function $g(s)=f(s)-s$, which can have a finite numbers of discontinuity and $\partial G(s)$ is the generalized gradient of $G$ at $s \in \mathbb{R}$.

\section{Preliminary results}

In this section, we have showed some technical lemmas that will be used in the next section to prove Theorem \ref{T1}.

\begin{lemma}\label{pr2}
	Let $u \in X$ satisfying $\Phi(u)>0$. Then, there exists a unique $t^{\ast}>0$ such that $u_{t^{\ast}} \in \mathcal{P}$. Hence, 
	$$
	I(u_{t^{\ast}})= \max_{t\geq 0} I(u_{t})
	$$
and for $u \in \mathcal{P}$,
$$
I(u)= \max_{t\geq 0} I(u_{t}).
$$
  
\end{lemma}

\begin{proof}
	Let $u \in X$ satisfying $\Phi(u)>0$. By $(F_{1})$, $u \neq 0$. For each $t\geq0$, we fix
	$$
	h(t):= I(u_{t}), \quad \forall t \in [0,+\infty).
	$$
	By $(X_1)-(X_2)$,
	$$
	h(t)= t^{\lambda_{1}} \psi_{1}(u)+ ... + t^{\lambda_{n}} \psi_{n}(u)- t^{\lambda_{\Phi}}\Phi(u).
	$$
	Now, From $(X_3)$ and $(F_{2})$, $h(t)>0$ for $t>0$ large small and 
	$$
	\lim_{t \rightarrow +\infty}h(t)=- \infty.
	$$
	These informations ensure that $h$ possesses a maximum in some $t^* \in (0, +\infty)$, that is, 
	$$
	I(u_{t^{\ast}})= \max_{t>0}I(u_{t}).
	$$ 
	Since $h'(t^{\ast})=0$, we derive that $u_{t^{\ast}} \in \mathcal{P}$.  To show the uniqueness of $t^{\ast}$, we first recall that 
	$
	u_{t} \in \mathcal{P}
	$
	if, and only if, 
	$$
	\lambda_{1}t^{\lambda_{1}}\psi_{1}(u)+...+\lambda_{n}t^{\lambda_{n}}\psi_{n}(u)=\lambda_{\Phi}t^{\lambda_{\Phi}}\Phi(u).
	$$
In the sequel, without loss of generality, we assume that $\lambda_{1}= \max \left\{ \lambda_{i}\right\}_{i=1}^{n}$. Then, $u_{t} \in \mathcal{P}$ if, and only if, 
	\begin{equation} \label{eq42}
		\lambda_{1}\psi_{1}(u) = - \sum_{i=2}^{n}\lambda_{i}t^{(\lambda_{i}-\lambda_{1})}\psi_{i}(u) + \lambda_{\Phi}t^{(\lambda_{\Phi} - \lambda_{1})}\Phi(u).
	\end{equation}
Combining  $(X_3)$ with $(F_{2})$ and the fact that  $\Phi(u)>0$, it follows that the function 
$$
m(t)=- \sum_{i=2}^{n}\lambda_{i}t^{(\lambda_{i}-\lambda_{1})}\psi_{i}(u) + \lambda_{\Phi}t^{(\lambda_{\Phi} - \lambda_{1})}\Phi(u), \quad \mbox{for} \quad t \geq 0
$$ 
has a positive derivative in $(0, + \infty)$. Therefore, $m$ is increasing, $m(0)=0$ and $m(t) \to +\infty$ as $t \to +\infty$. These facts guarantee the existence of a unique $t>0$ satisfying (\ref{eq42}).
\end{proof}

\begin{corollary}
The Pohozaev set $\mathcal{P}$ is not empty. 
\end{corollary}
\begin{proof}
	The result follows by combining the last lemma with $(F_{1})$. 
\end{proof}

\begin{lemma}\label{lm3}
	There exists $r>0$ such that  
	$$
	||u||\geq r, \quad \forall u \in \mathcal{P}.
	$$
\end{lemma}

\begin{proof}
	For all $u \in \mathcal{P}$,
	$$
	K(u)= \sum_{i=1}^{n}\lambda_{i}\psi_{i}(u) - \lambda_{\Phi}\Phi(u)=0.
	$$ 	
	On the other hand, by $(F_{3})$, there exists $r>0$ such that 
	$$
	\begin{array}{ll}
	K(u)&= \sum_{i=1}^{n}\lambda_{i}\psi_{i}(u) - \lambda_{\Phi}\Phi(u)>0, \quad \forall \,\, 0<||u||<r.
	\end{array} 
	$$
	Thus,  
	$$
	||u|| \geq r, \quad \forall u \in \mathcal{P}.
	$$
\end{proof}

\begin{proposition}\label{thm15}
The functional $I$ is bounded from below in $\mathcal{P}$ and there exists $u_{0}\in \mathcal{P}$ satisfying 
	$$
	I(u_{0})= \inf_{u \in \mathcal{P}}I(u).
	$$
Moreover, $\displaystyle \inf_{u \in \mathcal{P}}I(u)>0$.

\end{proposition}

\begin{proof}
	If $u \in \mathcal{P}$, $(F_{2})$ combined with $(X_3)$ gives 
	\begin{equation}\label{eq7}
		I(u)= J(u)-\Phi(u) = \sum_{i=1}^{n}\psi_{i}(u)-\sum_{i=1}^{n}\frac{\lambda_{i}}{\lambda_{\Phi}}\psi_{i}(u) = \sum_{i=1}^{n}\left(1-\frac{\lambda_{i}}{\lambda_{\Phi}}\right)\psi_{i}(u)\geq 0,
	\end{equation}
showing the boundedness of $I$ from below in $\mathcal{P}$. In what follows,
	$$
	I_{\infty}=\inf_{u \in \mathcal{P}}I(u)=\inf_{u \in \mathcal{P}^+}I(u)
	$$
	and $(u_k)$ is a minimizing sequence associated with $I_\infty$, that is, $(u_{k}) \subset \mathcal{P}^+$ and 
	$$
	I(u_{k})\rightarrow I_{\infty}.
	$$
	From (\ref{eq7}), $(J(u_{k}))$ is a bounded sequence. Hence, by $(F_{4})$, the sequence $(u_{k})$ is also bounded in $X$. On the other hand, by conditions $(X_6)-(X_7)$, we know that $(Q(u_{k})) \subset X^r$ and 
	$$
	J(Q(u_{k}))\leq J(u_{k}), \quad \forall k \in \mathbb{N}
	$$
	and
	$$
	\Phi(Q(u_{k})) \geq \Phi(u_{k})>0, \quad \forall k \in \mathbb{N}.
	$$ 
	By Lemma \ref{pr2}, there exits $t_{k}^{\ast}>0$ such that $(Q[u_{k}])_{t_{k}^{\ast}} \in \mathcal{P}$. Therefore,  
	$$
	I_{\infty} \leq I\left((Q[u_{k}])_{t_{k}^{\ast}}\right) \leq I\left((u_{k})_{t_{k}^{\ast}}\right)\leq \max_{t>0}I\left((u_{k})_{t}\right) = I(u_{k}).
	$$
The last inequality yields   
	$$
	I\left((Q[u_{k}])_{t_{k}^{\ast}}\right) \rightarrow I_{\infty}.
	$$
From this, without loss of generality we can assume that $(u_{k}) \subset X^r$.
	As ${X}$ is reflexive and $X^r$ is weak closed, we can suppose that for some subsequence, $(u_{k})$ is weakly convergent for some $u \in {X}^r$. Since,  $u_{k}\in \mathcal{P}$,  we must have $\Phi(u_{k})>0,$ and so, by $(F_{5})$, 
	$$
	\Phi(u)\geq 0.
	$$
We claim that $\Phi(u)>0$. Indeed, assume by contradiction that $\Phi(u)=0$. Then, 
	$$
	\Phi(u_{k})\rightarrow 0 \quad \mbox{as} \quad k \to +\infty.
	$$
Using the fact that $u_{k} \in \mathcal{P}$, we derive 
	$$
	J(u_{k}) \rightarrow 0 \quad \mbox{as} \quad k \to +\infty.
	$$
Now, applying $(F_{4})$ we get   
	$$
	||u_{k}|| \rightarrow 0,
	$$
	which contradicts Lemma \ref{lm3}. Thereby,  $\Phi(u)>0$, and so, $u \neq 0$. Consequently, the Lemma \ref{pr2} guarantees the existence of $t^{\ast}>0$ verifying 
	$$
	u_{t^{\ast}} \in \mathcal{P}.
	$$
From $(X_1)-(X_2)$,  
	$$
	\begin{array}{ll}
	I(u_{k}) &= \max_{t>0}I\left((u_{k})_{t}\right)\\
	&\geq I\left((u_{k})_{t^{\ast}}\right)\\
	& =\sum_{i=1}^{n}\psi_{i}((u_{k})_{t^{\ast}})- \Phi((u_{k})_{t^{\ast}}) \\
	& = \sum_{i=1}^{n}(t^{\ast})^{\lambda_{i}}\psi_{i}(u_{k})- (t^{\ast})^{\lambda_{\Phi}}\Phi(u_{k}).
	\end{array}
	$$
	The last inequality combines with $(F_5)-(F_{6})$ to give 
	$$
	\begin{array}{ll}
	I_{\infty} &\geq \liminf_{k \rightarrow \infty}\left(\sum_{i=1}^{n}(t^{\ast})^{\lambda_{i}}\psi_{i}(u_{k})- (t^{\ast})^{\lambda_{\Phi}}\Phi(u_{k})\right) \\
	& \geq \sum_{i=1}^{n}(t^{\ast})^{\lambda_{i}}\psi_{i}(u)- (t^{\ast})^{\lambda_{\Phi}}\Phi(u) \\
	& = I(u_{t^{\ast}}).
	\end{array}
	$$
Recalling that $u_{t^{\ast}} \in \mathcal{P}$, we deduce that  
	$$
	I_{\infty} = I(u_{t^{\ast}}).
	$$
Now, we are going to show that $I_\infty=\displaystyle \inf_{u \in \mathcal{P}}I(u)>0.$ In fact, by (\ref{eq7}), $\displaystyle I_\infty\geq0$. If $I_\infty = 0$, we can argue as above to find a minimizing sequence $(u_{k}) \subset \mathcal{P}^+$ satisfying $\Phi(u_k) \geq 0$ and $J(u_{k}) \rightarrow 0$. However, this information together with $(F_{4})$ leads to $||u_{k}|| \rightarrow 0$, contradicting Lemma \ref{lm3}. Therefore, $\displaystyle \inf_{u \in \mathcal{P}}I(u)>0,$ finishing the proof. 
\end{proof}

\section{Proof of Theorem \ref{T1} }
In this section our main goal is proving the Theorem \ref{T1}, however to do that, we need to prove more some preliminary lemmas.

\begin{lemma}\label{421}
	Let $u \in \mathcal{P}$ and set $\gamma(t) := u_{t}$. Then, 
	$$
	\lim\limits_{t \rightarrow +\infty} I(\gamma(t))=-\infty.
	$$
\end{lemma}

\begin{proof}
First of all, note that 
	\begin{equation}\label{eq12}
		I(\gamma(t))=\sum_{i=1}^{n}t^{\lambda^{i}}\psi_{i}(u)-t^{\lambda_{\Phi}}\Phi(u), \quad \forall t \in [0,+\infty). 
	\end{equation}
As $u \in \mathcal{P}$, we have $\Phi(u)>0.$ This combined with $(X_3)$  gives the desired result. 
\end{proof}
\begin{lemma}\label{422}
Let $\gamma:\mathbb{R} \rightarrow X$ be a continuous path satisfying 
$$
\gamma(0)=0 \quad \mbox{and} \quad \lim\limits_{t \rightarrow +\infty}I(\gamma(t))= - \infty.
$$
Then, there exists $t_{0}>0$ such that $\gamma(t_{0}) \in \mathcal{P}$.
\end{lemma}
\begin{proof}
We begin by supposing that $\gamma(t)\neq 0$ for all $t>0$.	Since  
	$$
	K(u)=\sum_{i=1}^{n}\lambda_{i}\psi_{i}(u) - \lambda_{\Phi}\Phi(u), 
	$$
by $(F_{3})$, there exists $r>0$ such that 
	$$
	K(u)>0, \quad \mbox{for} \quad ||u||<r. 
	$$
As $\gamma(0)=0$ and  $\gamma$, $K$ are continuous functions, for $t$ small enough we must have 
	$$
	K(\gamma(t))>0.
	$$ 
On the other hand, the definition of $I$, $(X_3)$ and  $(f_{3})$ lead to
 	$$
	K(u)=\lambda_{\Phi} I(u)+ \sum_{i=1}^{n}(\lambda_{i}-\lambda_{\Phi})\psi_{i}(u)\leq \lambda_{\Phi}I(u), \quad \forall u \in X.
	$$
	Hence,  
	$$
	K(\gamma({t})) \leq \lambda_{\Phi}I(\gamma({t})), \quad \forall t \in [0,+\infty).
	$$
	Then,  $K(\gamma(t))<0$ for $t$ large enough. From this, there is $\overline{t}>0$ verifying  
	$$
	K(\gamma(\overline{t}))=0,
	$$
	implying that $\gamma(\overline{t}) \in \mathcal{P}$. For the general case, fix $\tilde{t}>0$ satisfying
	$$
	\tilde{t}= \sup\left\{t\in [0, +\infty); \gamma(t)=0\right\}.
	$$ 
	By continuity, $\gamma(\tilde{t})=0$. Setting  
	$$
	\beta(t)=\gamma(t+\tilde{t}) \quad \forall t \in [0,+\infty),
	$$
	we have that $\beta(0)=0$, $\lim\limits_{t \rightarrow \infty}I(\beta(t))= - \infty$ and $\beta(t)\neq 0$, for all $t>0$. By the above arguments, there exists $t_{0}>0$ such that 
	$$
	\beta(t_{0}) \in \mathcal{P},
	$$
showing that $\gamma(t_{0}+ \tilde{t})\in \mathcal{P}$. 
\end{proof}

Next, we recall the deformation lemma for locally Lipschitz functional found in Figueiredo and Pimenta \cite{giovany} that will be used in the proof of Theorem \ref{T1}.
\begin{theorem}(Deformation Lemma)
	Let $X$ be a Banach space and $I:X \to \mathbb{R}$ be a locally 	Lipschitz functional. Assume that there are $c \in \mathbb{R}, S \subset E$ and $\alpha, \delta, \epsilon_{0}>0$ satisfying 
	$$
	\beta(x):=\min \left\{||z||_{X^{\ast}}; z \in \partial I(x) \right\} \geq \alpha, \mbox{ for all } x \in I^{-1}([c-\epsilon_{0}, c+\epsilon_{0}])\cap S_{2 \delta}.
	$$
where $S_{2\delta}$ is a $2\delta$-neighborhood of $S$. Then, for each $0<\epsilon < \min \{\frac{\delta \alpha}{2}, \epsilon_{0}\}$ there exists a homeomorphism $\eta: X \rightarrow X$ satisfying 
	\begin{itemize}
		\item $\eta(u)=u$, se $u \notin I^{-1}([c-\epsilon_{0}, c+\epsilon_{0}]) \cap S_{2\delta} $;
		\item $\eta(I^{c+\epsilon} \cap S) \subset I^{c-\epsilon}$;
		\item $I(\eta(u)) \leq I(u)$ for all $u \in X$,
	\end{itemize}
where $I^{a}=\{u \in X;I(u)\leq a\}$ for all $a \in \mathbb{R}$.
\end{theorem}

Now, we are ready to proof Theorem \ref{T1}.

\begin{proof}({\bf Proof of Theorem \ref{T1} }) 
	By Proposition \ref{thm15}, there is $u \in \mathcal{P}$ with 
	$$
	I(u)=I_\infty= \inf_{w \in \mathcal{P}}I(w)>0.
	$$ 
	Assume by contradiction that $0 \notin \partial I(u)$. Then, there is $\alpha>0$ such that 	
	\begin{equation} \label{ob424} 
			|x-u|< \alpha \Rightarrow	\beta(x)=\min \left\{||z||_{X^{\ast}}; z \in \partial I(x) \right\}> \alpha. 
	\end{equation}
	Indeed, otherwise for each $\alpha=\frac{1}{k}$, it would exist $x_{k} \in X$ with 
		$$
		|u-x_{k}|< \frac{1}{k} \quad \mbox{ and } \quad \beta(x_{k})< \frac{1}{k}, \quad \forall k \in \mathbb{N}.
		$$
		Consequently, it would exist $z_{k} \in \partial I(x_k)$ with
		$$
		||z_{k}||_{X^{\ast}}< \frac{1}{k}.
		$$
		By definition and properties of $I^{0}(u,v)$ (see \cite{Chang}), we must have 
		$$
		I^{0}(u,v) \geq \limsup_{k \rightarrow \infty}I^{0}(x_{k},v) \geq \limsup_{k \rightarrow \infty} \langle z_{k},v \rangle  = \langle 0,v \rangle , \quad \forall v \in X,
		$$
		showing that  $0 \in \partial I(u)$, which is absurd. This proves (\ref{ob424}).  
		
		Applying the Deformation Lemma  for $c=I_\infty, \delta=\frac{\alpha}{4}$, $\epsilon_{0} = \frac{I_\infty}{2}$ and $S=B_{\frac{\alpha}{2}}(u)$, we get a homeomorphism $\eta:X \to X$ satisfying 
	\begin{itemize}
		\item $(i)$ \,\,  $\eta(u)=u$, if $u \notin I^{-1}([c-\epsilon_{0}, c+\epsilon_{0}]) \cap S_{2\delta} $;
		\item $(ii)$ \,\, $\eta(I^{c+\epsilon} \cap S) \subset I^{c-\epsilon}$;
		\item $(iii)$ \,\,  $I(\eta(u)) \leq I(u)$ for all $u \in X$.
	\end{itemize}
Fix
	$$
	\beta(t):=\eta(\gamma(t))
	$$
where $\gamma(t)=u_{t}$. By $(X_5)$, the function $\beta$ is continuous. Moreover, from $(X_4)$ and $I(0)<c- \epsilon_{0},$ we obtain
	$$
	\beta(0)=\eta(\gamma(0)) = \eta(0)=0.
	$$
Now, by Lemma \ref{421} and $(iii)$,
	\begin{equation}\label{eq46}
		\lim\limits_{t \rightarrow +\infty}I(\beta(t)) =\lim\limits_{t \rightarrow +\infty}I(\eta(\gamma(t)))\leq \lim\limits_{t \rightarrow +\infty} I(\gamma(t)) = - \infty.
	\end{equation}
The above analysis permit us to apply Lemma \ref{422} to find $t^{\ast}>0$ such that $\beta(t^{\ast}) \in \mathcal{P}$. Hence,
	\begin{equation}\label{eq47}
		c \leq I(\beta(t^{\ast})) \leq \max_{t>0}I(\beta(t)).
	\end{equation}
On the hand, as $I(u)=c<c+\epsilon$, $\gamma(1)=u$, $I$ and $\gamma$ are continuous, we can choose $\tau>0$ of such way that 
$$
\gamma(t)\in I^{c+\epsilon}\cap S, \quad \forall  t\in [1-\tau,1+\tau].
$$
Thereby, if $t\in [1-\tau,1+\tau]$, the Deformation Lemma yields 
$$
	I(\beta(t)) = I(\eta(\gamma(t)))\leq c-\epsilon.
$$
Now, we will analyze the case $t\notin [1-\tau,1+\tau]$. In this case, by Lemma  \ref{pr2} and $(iii)$,  
	$$
	I(\beta(t)) = I(\eta(\gamma(t)))\leq I(\gamma(t))<\max_{t>0}I(\gamma(t)) = I(\gamma(1))=c
	$$
In any case, we deduce that 
$$
		\max_{t>0} I(\beta(t)) <c,
$$ 
contradicting (\ref{eq47}). Therefore, 	
$$
	0 \in \partial I(u),
$$
finishing the proof. 
\end{proof}

The Corollary below is a version of Theorem \ref{T1} when the functional $I$ is $C^{1}(X, \mathbb{R})$.

\begin{corollary}\label{cor426}
Let $X,\psi_{1}$,..., $\psi_{n}$ and $\Phi$  satisfying $(X_{1})-(X_8)$ and $(F_1)-(F_6)$. Assuming that $\Phi$, $\psi_{1}$,..., $\psi_{n} \in C^{1}(X, \mathbb{R})$ and 
$$
\inf_{w \in \mathcal{P}}I(w)=\inf_{w \in \mathcal{P^+}}I(w).
$$
Then there exists $u \in \mathcal{P}$ such that 
$$
I(u)= \inf_{w \in \mathcal{P}}I(w).
$$
Moreover, $u$ is a critical point of $I$ in $X$, that is, $I'(u)=0.$
\end{corollary}

\begin{proof}
Since $I \in C^{1}(X, \mathbb{R})$, we have that 
	$$
	\partial I (u) = \left\{I'(u)\right\}, \quad \forall u \in X. 
	$$
Therefore, o corollary is an immediate consequence of the Theorem \ref{T1}. 
\end{proof}

\section{Problem involving $s$ and $t$ fractional laplacian for $0<t,s<1$} 
As mentioned in the introduction, in this section we intend to prove the existence of nonnegative solution for a problem like
$$
\sum_{j=1}^{n}(- \Delta)^{s_j}u=g(u), \quad \mbox{in} \quad \mathbb{R}^N, \eqno{(P_1)}
$$
where $0<s_1\leq s_2\leq  .... \leq s_n<1, N > 2s_n$ and $(- \Delta)^{s_i}$ denotes the $s_i$-fractional laplacian and 
$$
g(s)=f(s)-s, \quad \forall s \in \mathbb{R}, 
$$ 
with $f:\mathbb{R} \rightarrow \mathbb{R}$ being a continuous function satisfying: \\

\noindent \,\,  $(f_1)$ \,\, $\displaystyle \lim_{s \rightarrow 0} \frac{f(s)}{s}=0$. \\

\noindent \,\, $(f_2)$ \,\, $\displaystyle \limsup_{s \rightarrow +\infty} \frac{|f(s)|}{|s|^{q-1}}< \infty$, for some $q \in (1, 2^{\ast}_{s_{n}}-1)$ where 
$ 2^{\ast}_{s_{n}}= \frac{2N}{N-2s_{n}}.$ \\

\noindent \,\,  $(f_3)$ \,\, $f(s)>0,$ \quad $\forall s>0$. \\

\noindent \,\, $(f_4)$ \,\, There is $\tau>0$ such that $G(\tau)=\int_{0}^{\tau}g(s)ds>0$. \\

Since we intend to find a nonnegative solution, in what follows we assume that 
$$
f(s)=0, \quad \forall s<0, 
$$
and denote by $F$ its primitive, that is, 
$$
F(s)= \int_{0}^{s}f(t)dt.
$$

The energy functional associated with $(P_1)$ is given by $I_1:H^{s_{n}}(\mathbb{R}^{N})\rightarrow \mathbb{R}$ with
$$
I_1(u)=\sum_{i=1}^{n}\frac{1}{2}\int_{\mathbb{R}^{N}}\int_{\mathbb{R}^{N}} \frac{|u(x)-u(y)|^{2}}{|x-y|^{N+2s_{i}}}dx\,dy-\int_{\mathbb{R}^{N}}G(u)dx.
$$
It is easy to see that $I_1 \in C^{1}(H^{s_n}(\mathbb{R}^N),\mathbb{R})$ and its critical points are weak solutions of $(P_1)$. 

We recall that, for any $s \in (0,1)$, the fractional Sobolev space $H^{s}(\mathbb{R}^N)$ is defined by
\[
H^{s}(\mathbb{R}^N)=\Big\{u\in L^2(\mathbb{R}^N): \ \int_{\mathbb{R}^{N}}\int_{\mathbb{R}^{N}}\frac{|u(x)-u(y)|^2}{|x-y|^{N+2s}}dx\,dy<\infty\Big\},
\]
endowed with the norm
$$
\|u\|=\Big(|u|_{L^2(\mathbb{R}^{N})}^2+\int_{\mathbb{R}^{N}}\int_{\mathbb{R}^{N}}\frac{|u(x)-u(y)|^2}{|x-y|^{N+2s}}dx\,dy\Big)^{1/2}.
$$
The fractional  Laplacian, $(-\Delta)^{s}u,$  of a smooth function $u:\mathbb{R}^{N} \rightarrow  \mathbb{R}$ is defined  by  
$$
{\mathcal F}((-\Delta)^{s}u)(\xi)=|\xi|^{2s}{\mathcal F}(u)(\xi), \ \xi \in \mathbb{R}^N,
$$ 
where ${\mathcal F}$ denotes the Fourier transform, that is, 
\[
{\mathcal F}(\phi)(\xi)=\frac{1}{(2\pi)^{\frac{N}{2}}} \int_{\mathbb{R}^N} \mathit{e}^{-i \xi \cdot x} \phi (x)  \, 
\ d x \equiv \widehat{\phi}(\xi) ,  
\]
for functions $\phi$ in  the  Schwartz class. As mentioned in \cite[Lemma 3.2]{nezza}, $(-\Delta)^{s}u$  can be equivalently represented  by
$$
(-\Delta)^{s} u(x) = -\frac{1}{2} C(N,s)\int_{\mathbb{R}^N}\frac{(u(x+y)+u(x-y)-2 u(x))}{|y|^{N+2s}}\ d y, \ \forall x \in \mathbb{R}^N,
$$
where 
$$C(N,s)=(\int_{\mathbb{R}^N}\frac{(1- cos\xi_1)}{|\xi|^{N+2s}}d\xi)^{-1},\  \xi=(\xi_1,\xi_2,\ldots,\xi_N).$$
Also, in light of \cite[Propostion~3.4,Propostion~3.6]{nezza}, we have
\begin{equation}
\label{equinorm}
|(-\Delta)^{s/2} u|^2_{L^2(\mathbb{R}^N)}=\int_{\mathbb{R}^N}|\xi|^{2s}|\widehat{u}|^2d\xi=\frac{1}{2}C(N,s)
\int_{\mathbb{R}^{N}}\int_{\mathbb{R}^{N}}\frac{|u(x)-u(y)|^2}{|x-y|^{N+2s}}dx\,dy,
\end{equation}
for all $u\in H^{s}(\mathbb{R}^N)$, and sometimes, we identify these two quantities by omitting the normalization constant $\frac{1}{2} C(N,s).$
For $ N > 2s,$  from \cite[Theorem 6.5]{nezza} we also know that, for any $p \in [ 2, 2^{*}_{s}]$,
there exists $C_p>0$ such that
\begin{equation}
\label{emb}
|u|_{L^p(\mathbb{R}^{N})}\leq C_p\|u\|,
\,\quad \mbox{for all $u\in H^{s}(\mathbb{R}^N)$}.
\end{equation}

In the sequel, we will work to show that functional $I_1$ verifies the assumptions of Theorem \ref{T1}. To this end, we need to fix some notations: 

$$
\psi_i(u)=\frac{1}{2}\int_{\mathbb{R}^{N}}\int_{\mathbb{R}^{N}} \frac{|u(x)-u(y)|^{2}}{|x-y|^{N+2s_{i}}}dx\,dy \quad \mbox{for} \quad i\in\{1,...,n\}, 
$$
$$
\Phi(u)=\int_{\mathbb{R}^{N}}G(u)dx=\int_{\mathbb{R}^{N}}F(u)dx-\frac{1}{2}\int_{\mathbb{R}^{N}}|u|^{2}dx,
$$
$$
X=H^{s_n}(\mathbb{R}^N), \quad X^{+}=\{u \in H^{s_n}(\mathbb{R}^N)\,:\, u(x) \geq 0 \quad \mbox{a.e. in} \quad \mathbb{R}^N\}
$$
and
$$
X^r=\{u \in H_{rad}^{s_n}(\mathbb{R}^N) \cap X^+\,:\, 0 \leq u(x) \leq u(y) \quad \mbox{if} \quad 0<|y| \leq |x| \}.
$$

The reader is invited to observe that $X^+$ and $X^r$ are weak closed in $H^{s_n}(\mathbb{R}^N)$. Moreover, we would like point out that $X^r$ is compactly embedding $L^{q}(\mathbb{R}^N)$ for all $q \in (2,2^{*}_{s_n}) $, that is, if $(u_k) \subset X^r$ is a bounded sequence in $H^{s_n}(\mathbb{R}^N)$, then there are a subsequence of $(u_k)$, still denoted by itself, and $u \in X^r$ such that 
$$
u_k \to u \quad \mbox{in} \quad L^{q}(\mathbb{R}^N), \quad \forall q \in (2,2^{*}_{s_n}).
$$
The proof this fact follows of \cite[Radial Lemma A.IV ]{berest}.

The main result this section has the following statement
\begin{theorem} Assume the conditions $(f_1)-(f_4)$. Then, $(P_1)$ has a nontrivial solution. 
\end{theorem}

\begin{proof} In the sequel, we will show that all conditions of Theorem \ref{T1} hold, which ensures that functional $I_1$ has a nontrivial critical point, and hence $(P_1)$ has a nontrivial solution. First of all, note that
$$
I_1(u)=J(u)-\Phi(u), \quad \forall u \in X=H^{s_{n}}(\mathbb{R}^{N}). 
$$	
	
In what follows, we set $\ast(t,u):=u_{t}: \mathbb{R}^N \to \mathbb{R}$ by 
	$$
	u_{t}(x)=
	\left\{
	\begin{array}{l}
	u \left(\frac{x}{t}\right), \quad \mbox{ for } \quad t > 0, \\
	0, \quad \mbox{for} \quad t=0.
	\end{array}
	\right.
	$$	
It is easy to check that 
$$
\psi_{1}	,\psi_{2},..., \psi_{n},\Phi \in C^{1}(H^{s_{n}}(\mathbb{R}^{N}), \mathbb{R}),
$$
$$
u \in X^r \Rightarrow u_t \in X^r, \quad \forall t \geq 0,
$$
$$
\Phi(u_{t}) = t^{N}\Phi(u), \quad  \forall t\geq 0 \quad \mbox{and} \quad \forall u \in H^{s_{n}}(\mathbb{R}^{N})
$$
and
$$
\psi_{i}(u_{t})= t^{N-2s_{i}}\psi_{i}(u), \quad \forall t\geq 0, \quad \forall i \in \{1,2,...,n\}  \quad \mbox{and} \quad \forall u \in H^{s_{n}}(\mathbb{R}^{N}).
$$ 

Thus, the conditions $(X_1)-(X_4)$ and $(F_2)$ occur. Moreover, a simple computation shows that for each $u \in H^{s_{n}}(\mathbb{R}^{N})$, the application $t \longmapsto u_{t}$ is continuous, and so, $(X_5)$ is also proved.

The conditions $(X_6)-(X_8)$ are verified by considering 	
	$$
	\begin{array}{cccl}
	Q:& H^{s_{n}}(\mathbb{R}^{N})& \longrightarrow& X^r \\
	&u& \longmapsto &(u^{+})^{\ast}	
	\end{array}
	$$
	where $(u^{+})^{\ast}$ is the Schwartz's symmetrization of $u^{+}=\max\{u,0\}$. 

\vspace {0.5 cm}

Now, we are going to prove the conditions $(F_1)$ and $(F_3)-(F_6)$.

	\begin{claim}\label{4302}( {\bf Proof of $(F_1)$} ):  \, 
		There exists $u \in H^{s_{n}}(\mathbb{R}^{N})$ such that $\Phi(u)>0$ and $\Phi(0)=0$.
	\end{claim}
	\begin{proof}
		By definition of $\Phi$, we have $\Phi(0)=0$. For each $k \in \mathbb{N}$, take  $\phi_{k}\in C_{0}^{\infty}(B_{1+\frac{1}{k}}(0))$ with $\phi_{k} = \tau$ in $B_{1}(0)$ and  $|\phi_{k}| \leq \tau$. Note that 
		$$
		\int_{\mathbb{R}^{N}}G(\phi_k)dx = \int_{B_{1}(0)}G(\phi_k)dx+ \int_{B_{1+\frac{1}{k}}(0)\setminus B_{1}}G(\phi_k)dx
		$$
		and
		$$
		\left|\int_{B_{1+\frac{1}{k}}(0)\setminus B_{1}(0)}G(\phi_k)dx \right| \leq \left(\sup_{\left\{|t|\leq \tau\right\}}G(t)\right)|B_{1+\frac{1}{k}}\setminus B_{1}(0)|.
		$$
	As
		$$
		|B_{1+\frac{1}{k}}\setminus B_{1}(0)| \rightarrow 0 \quad \mbox{as} \quad k \to +\infty, 
		$$
	and  $G(\tau)>0$, we can fix $k \in \mathbb{N}$ large enough such that 
			$$
		\begin{array}{ll}
	\displaystyle 	\int_{\mathbb{R}^{N}}G(\phi_k)dx &= 	\displaystyle \int_{B_{1}(0)}G(\phi_k)dx+ \int_{B_{1+\frac{1}{k}}(0)\setminus B_{1}}G(\phi_k)dx\\
		& \geq G(\tau)|B_{1}(0)|-\left(\sup_{\left\{|t|\leq \tau\right\}}G(t)\right)|B_{1+\frac{1}{k}}\setminus B_{1}(0)|, \\
		&> 0,
		\end{array}
		$$
	proving the claim. 
	\end{proof}

	\begin{claim}\label{4303}( {\bf Proof of $(F_3)$})
		There exists  $r>0$ such that  
		$$
		\sum_{i=i}^{n}\lambda_{i}\psi_{i}(u)>\lambda_{\Phi}\Phi(u), \quad \mbox{for} \quad 0<||u||< r.
		$$ 
	\end{claim}
	\begin{proof}
	Given  $\epsilon< \frac{1}{2}$, we find $C_{\epsilon}>0$ verifying 
		$$
		\begin{array}{ll}
		\lambda_{n}\psi_{n}(u) - \lambda_{\Phi} \Phi(u) &= \lambda_{n}\psi_{n}(u) + \frac{\lambda_{\Phi}}{2}\displaystyle \int_{\mathbb{R}^{N}}|u|^{2}dx - \lambda_{\Phi}\int_{\mathbb{R}^{N}}F(u)dx\\
		&\geq \lambda_{n}\psi_{n}(u) + \lambda_{\Phi}(\frac{1}{2}- \epsilon)\displaystyle \int_{\mathbb{R}^{N}}|u|^{2}dx - \lambda_{\Phi} C_{\epsilon}\int_{\mathbb{R}^{N}}|u|^{q}dx \\
		& \geq C||u||^{2}-C_{1}||u||^{q},
		\end{array}
		$$
		where $C,C_1>0$ are positive constants. As $q>2$, we obtain the desired result. 
	\end{proof}
	
	\begin{claim}\label{4304}( {\bf Proof of $(F_4)$})
	Let $(u_{k})$ be a sequence in $H^{s_{n}}(\mathbb{R}^{N})$ with $\Phi(u_{k})\geq 0$ for all $k \in \mathbb{N}$. If $(J(u_{k}))$ is bounded,  we have that  $(u_{k})$ is also bounded. Moreover, if $J(u_{k})\rightarrow 0$, then $||u_{k}|| \rightarrow 0$. 
	\end{claim}
	\begin{proof}
		Fix $u \in H^{s_n}(\mathbb{R}^N)$ with $\Phi(u)\geq 0$. By assumptions on $f$, there exists $C_{\frac{1}{4}}>0$ satisfying  
		$$
		\int_{\mathbb{R}^{N}}F(u)dx \leq \frac{1}{4}\int_{\mathbb{R}^{N}}|u|^{2}dx + C_{\frac{1}{4}}\int_{\mathbb{R}^{N}}|u|^{2^{\ast}_{s_{n}}}dx.
		$$
		As $\Phi(u)\geq 0$, 
		$$
		\begin{array}{ll}
		\frac{1}{2}\displaystyle \int_{\mathbb{R}^{N}}|u|^{2}dx & \leq \Phi(u) + \frac{1}{2}\displaystyle\int_{\mathbb{R}^{N}}|u|^{2}dx \\
		& = \displaystyle\int_{\mathbb{R}^{N}}F(u)dx \\
		& \leq   \frac{1}{4}\displaystyle\int_{\mathbb{R}^{N}}|u|^{2}dx + C_{\frac{1}{4}}\displaystyle\int_{\mathbb{R}^{N}}|u|^{2^{\ast}_{s_n}}dx.
		\end{array}	
		$$
		From this, 
		$$
		\frac{1}{4}\int_{\mathbb{R}^{N}}|u|^{2}dx \leq C_{1}[\psi_{n}(u)]^{\frac{2^{\ast}_{s_{n}}}{2}} \leq C_{2}[J(u)]^{\frac{2^{\ast}_{s_{n}}}{2}},
		$$
	for some positive constants $C_1,C_{2}>0$. Therefore, 
		$$
		||u||^{2} \leq J(u) + C_{3}[J(u)]^{\frac{2^{\ast}_{s_{n}}}{2}},
		$$
	showing the Claim \ref{4304}.
	\end{proof}	
	\begin{claim}\label{a4305}( {\bf Proof of $(F_5)$})
		If $(u_{k})$ is weakly convergent for $u$ in $X^r$, then 
		$$
		\limsup_{k \rightarrow \infty}\Phi(u_{k}) \leq \Phi(u).
		$$
	\end{claim}
	
	\begin{proof}
	The assumptions on $f$ ensure that for each $\epsilon>0$, there is $C_{\epsilon}>0$ satisfying 
		$$
		F(s) \leq \frac{\epsilon}{6L} |s|^{2} + C_{\epsilon}|s|^{p}.
		$$	
	where $L=\sup_{k \in \mathbb{N}}||u_{k}||^{2}$. Then, for $R>0$
		$$
		\int_{B_{R}^{c}}|F(u_{k})|dx \leq  \frac{\epsilon}{6} +  C_{\epsilon}\int_{B_{R}^{c}}|u_{k}|^{p}dx.
		$$
		Since $X^r$ is compactly embedding in $L^{q}(\mathbb{R}^{N})$ for $q \in (2, 2^{\ast})$, there are $R$ and $k_0$ large enough such that 
		$$
		\int_{B_{R}^{c}}|F(u)|dx <\frac{\epsilon}{3} \mbox{ and } \int_{B_{R}^{c}}|F(u_k)|dx <\frac{\epsilon}{3}, \quad \forall k \geq k_0.
		$$
	Now, as $H^{s_{n}}(\mathbb{R}^{N})$ is compactly embedding in  $L^{q}(B_{R})$ for $p \in [1,2^{\ast}_{n})$ and $f$ has subcritical growth, we have 
		$$
		\int_{B_{R}}|F(u_{k})-F(u)|dx \rightarrow 0.
		$$
	From this, 
		$$
	\limsup_{k \rightarrow \infty}\int_{\mathbb{R}^{N}}|F(u_{k})-F(u)|dx \leq \epsilon, 
		$$
implying that 
		$$
		\int_{\mathbb{R}^{N}}F(u_{k})dx \rightarrow \int_{\mathbb{R}^{N}}F(u)dx.
		$$
The last limit leads to
		$$
		\begin{array}{ll}
		\liminf \left(\displaystyle \frac{1}{2} \int_{\mathbb{R}^{N}}|u_{k}|^{2}dx -\int_{\mathbb{R}^{N}}F(u_{k})dx\right) &\geq \liminf \left(\displaystyle \frac{1}{2} \int_{\mathbb{R}^{N}}|u_{k}|^{2}dx\right)-\displaystyle \int_{\mathbb{R}^{N}}F(u)dx\\
		& \geq \displaystyle \frac{1}{2}\int_{\mathbb{R}^{N}}|u|^{2}dx-\int_{\mathbb{R}^{N}}F(u)dx = -\Phi(u),
		\end{array}
		$$
	that is, $\displaystyle \liminf_{k \to +\infty} (-\Phi(u_{k})) \geq - \Phi(u)$, proving the result. 
	\end{proof}
	
	\begin{claim}\label{a4306}( {\bf Proof of $(F_6)$})
		Se $(u_{k})$ is weakly convergent to $u$ in $H^{s_{n}}(\mathbb{R}^{N})$, then   
		$$
		 \psi_{i}(u) \leq \liminf_{k \rightarrow \infty}\psi_{i}(u_{k}), \,\, \forall i \in \{1,2,...,n\}.
		 $$
	\end{claim}	
	\begin{proof}
	The result is an immediate consequence of the fact that $\psi_i$ is a convex function in $H^{s_{n}}(\mathbb{R}^{n})$ for $i \in \{1,2,...,n\}$. 
	\end{proof}
	\begin{claim} \label{igualdade} The functional $I_1$ verifies the equality below  
	$$
	\inf_{w \in \mathcal{P}}I_1(w)=\inf_{w \in \mathcal{P}^+}I_1(w).
	$$ 	
	\end{claim}
\begin{proof} By definition of $I_1$, it is easy to see that 
$$
I_1(u) \geq I_1(u^{+}), \quad \forall u \in H^{s_n}(\mathbb{R}^{N}),
$$	
where $u^{+}=\max\{u,0\}.$  For each $u \in \mathcal{P}$, we know that $u^{+} \not=0$, thus there is $t^{+}>0$ such that $(u^{+})_{t^{+}} \in \mathcal{P}$. Then, 
$$
\inf_{w \in \mathcal{P}^+}I_1(w) \leq I_1((u^{+})_{t^{+}}) \leq I_1((u)_{t^{+}}) \leq \max_{t>0}I_1(u_t)=I_1(u),\quad \forall u \in \mathcal{P},
$$
showing the desired result. 
\end{proof}

The above claims permit us to conclude that $I_1$ verifies the assumptions of Theorem \ref{T1}, more precisely Corollary \ref{cor426}. Hence, $I_1$ has a nontrivial critical point, and so, problem $(P_1)$ possesses a nontrivial solution.  
\end{proof}

Before concluding this section, we would like point out that the reader can find recent results involving fractional Laplacian in Barrios, Colorado, de Pablo and S\'anchez \cite{barrios}, Br\"andle, Colorado and  S\'anchez \cite{Brandle},  Cabr\'e and Sire \cite{cabre}, Caffarelli and Silvestre \cite{caffarelli}, Fall, Mahmoudi and  Valdinoci  \cite{Moustapha}, Felmer, Quass and Tan \cite{FQT},  Secchi \cite{Secchi} and their references.

\section{ Existence of solution for a class of anisotropic problem} 

In this section we study the existence of solution for the following anisotropic problem
$$ 
-\sum_{i=1}^{N}\frac{\partial}{\partial x_i}\left(\left|\frac{\partial u}{\partial x_i} \right|^{p_i-2}\frac{\partial u}{\partial x_i}\right)=g(u), \quad \mbox{in} \quad \mathbb{R}^N, \eqno{(P_2)}
$$
where $1< p_{1} \leq ... \leq p_{N}<N$ and $g:\mathbb{R} \to \mathbb{R}$ is a function given by 
$$
g(s)=f(s)-|s|^{p_1-2}s, \quad \forall s \in \mathbb{R}, 
$$  
with $f:\mathbb{R} \rightarrow \mathbb{R}$ being a continuous function satisfying $(f_3)-(f_4)$ and the conditions: \\

\noindent $(f_5)$ \,\,  $\displaystyle \lim_{s \rightarrow 0} \frac{f(s)}{|s|^{p_1-1}}=0$. \\
\noindent $(f_6)$ \,\, $\displaystyle \limsup_{s \rightarrow +\infty} \frac{|f(s)|}{|s|^{q-1}}< \infty$, for some $q \in (p_0, p^{\ast})$ where 
	$$
	p^{\ast}= \frac{N}{\sum_{i=1}^{N}\frac{1}{p_{i}}-1}.
	$$
As in the previous section, we will assume that 
$$
f(s)=0, \quad \forall s<0, 
$$
and denote by $F$ its primitive, that is, 
$$
F(s)= \int_{0}^{s}f(t)dt.
$$

The main theorem in this section is the following 
\begin{theorem} \label{ANST}
	Assume the conditions $(f_3)-(f_6)$. Then, $(P_2)$ has a nontrivial solution. 
\end{theorem}

The reader can find some results associated with anisotropic problems in Alves and El Hamidi \cite{AH}, El Hamidi and Rakotoson, \cite{HR1,HR2,HR3},  Fragala, Gazzola, and Kawohl \cite{FGK} and their references.

Hereafter, we fix $\overrightarrow{p}=(p_{1},...,p_{N})$  and define the anisotropic Sobolev space $W^{1,\overrightarrow{p}}(\mathbb{R}^{N})$ by    
$$
W^{1,\overrightarrow{p}}(\mathbb{R}^{N}) = \left\{u \in L^{p_{1}}(\mathbb{R}^{N}); \frac{\partial u}{\partial x_{i}} \in L^{p_{i}}(\mathbb{R}^{N}) \,\,\, i=1,2,...,N \right\}
$$
endowed with the norm 
$$
||u||:=\left(\int_{\mathbb{R}^{N}}|u|^{p_{1}}\, dx \right)^{\frac{1}{p_{1}}} + \sum_{i=1}^{N}\left(\int_{\mathbb{R}^{N}}\left|\frac{\partial u}{\partial x_{i}}\right|^{{p_{i}}}\,dx \right)^{\frac{1}{p_{i}}}.
$$ 

Related to the space $(W^{1,\overrightarrow{p}}(\mathbb{R}^{N}), ||\,\,\,||)$, it is possible to prove that it is a Reflexive Banach space and $C_{0}^{\infty}(\mathbb{R}^{N})$ is dense in $W^{1,\overrightarrow{p}}(\mathbb{R}^{N})$. Moreover, the space $W^{1,\overrightarrow{p}}(\mathbb{R}^{N})$ is continuously embedding in $L^{q}(\mathbb{R}^{N})$ for all  $q \in [p_{1}, p^{\ast}]$. For more details about this subject see Nikol'skii \cite{Niko} and Rakosnik \cite{Rak1, Rak2}.

\vspace{0.5 cm}

The proof of the next lemma follows the same ideas explored in \cite[Lemma 2.5 and Theorem 3.1]{Hajaiej} and also its proof will be omit. 
\begin{lemma} \label{radial}
	If $u \in W^{1, \overrightarrow{p}}(\mathbb{R}^{N})$ is a nonnegative function, then 
	$$
\left|\left|\frac{\partial u^{\ast}}{\partial x_{i}}\right|\right|^{p_i}_{p_{i}} \leq \left|\left|\frac{\partial u}{\partial x_{i}}\right|\right|^{p_i}_{p_{i}}, \quad \forall i \in \{1,2,...,n\}, 
	$$
where  $u^*$ is the Schwartz's symmetrization of $u$.
\end{lemma}

Our intention is proving that the energy functional $I_2: W^{1, \overrightarrow{p}}(\mathbb{R}^{N}) \to \mathbb{R}$ given by
$$
I_2(u)= \sum_{i=1}^{N}\frac{1}{p_i}\int_{\mathbb{R}^{N}}\left|\frac{\partial u}{\partial x_{i}}\right|^{p_{i}}dx +\int_{\mathbb{R}^{N}}|u|^{p_{1}}dx -\int_{\mathbb{R}^{N}}F(u)dx
$$
satisfies the conditions of Theorem \ref{T1}, because the critical points of $I_2$ are weak solutions of $(P_2)$. Since $I_2$ belongs to $C^{1}(W^{1, \overrightarrow{p}}(\mathbb{R}^{N}),\mathbb{R})$ we will prove that $I_2$ verifies the conditions of Corollary \ref{cor426}. Having this in mind, in what follows we define 
$\psi_{i}, \Phi:W^{1,\overrightarrow{p}}(\mathbb{R}^{N}) \rightarrow \mathbb{R}$ by 
$$
\psi_{i}(u)= \frac{1}{p_i}\int_{\mathbb{R}^{N}}\left|\frac{\partial u}{\partial x_{i}}\right|^{p_{i}}dx, \quad \mbox{for} \quad i \in \{1,2,...,N\}
$$
$$
\Phi(u)=\int_{\mathbb{R}^{N}}G(u)dx=\int_{\mathbb{R}^{N}}F(u)dx-\frac{1}{p_1}\int_{\mathbb{R}^{N}}|u|^{p_{1}}dx
$$
$$
X= W^{1, \overrightarrow{p}}(\mathbb{R}^{N}), \quad X^{+}=\{u \in  W^{1, \overrightarrow{p}}(\mathbb{R}^{N})\,:\, u(x) \geq 0 \quad \mbox{a.e. in} \quad \mathbb{R}^N\}
$$
and
$$
X^r=\{u \in  W_{rad}^{1, \overrightarrow{p}}(\mathbb{R}^{N}) \cap X^+\,:\, 0 \leq u(x) \leq u(y) \quad \mbox{if} \quad 0<|y| \leq |x| \}.
$$

Arguing as \cite[Radial Lemma A.IV ]{berest} it is possible to prove that $X^r$ is compactly embedding $L^{q}(\mathbb{R}^N)$ for all $q \in (p_1,p^{*}). $

From the above notations 
$$
I_2(u)=J(u)-\Phi(u), \quad \forall u \in X= W^{1, \overrightarrow{p}}(\mathbb{R}^{N}).
$$ 

As in the previous section, we also consider $\ast(t,u):=u_{t}: \mathbb{R}^N \to \mathbb{R}$  by 
$$
u_{t}(x)=
\left\{
\begin{array}{l}
u \left(\frac{x}{t}\right), \quad \mbox{ for } \quad t > 0, \\
0, \quad \mbox{for} \quad t=0.
\end{array}
\right.
$$	
A simple computation gives,
$$
\psi_i(u_t) = t^{N-p_{i}}\psi_{i}(u) \quad \mbox{and} \quad \Phi(u_t) = t^{N}\Phi(u), \quad \forall t \geq 0 \quad \mbox{and} \quad \forall u \in W^{1, \overrightarrow{p}}(\mathbb{R}^{N}).
$$
Moreover,  the application $t  \mapsto u_{t}$ is a continuous function in $t \in [0, +\infty)$ for all $u \in W^{1,\overrightarrow{p}}(\mathbb{R}^{N})$ and 
$$
u_t \in X^r \quad \forall t \in \mathbb{R} \quad \mbox{when} \quad u \in X^r.
$$
In the sequel, we set $Q:W^{1, \overrightarrow{p}}(\mathbb{R}^{N}) \longrightarrow X^r $ by 
$$
Q(u)=(u^{+})^{\ast}.
$$
By using the Lemma \ref {radial} and properties of radial functions, it follows that
$$
\psi_{i}(Q(u)) \leq \psi_{i}(u) \quad \mbox{and} \quad \Phi(u) = \Phi(Q(u)), \quad \forall u \in X^+.
$$
Moreover, with few modifications, we can argue as in Section 4 to show the claims below:
\begin{itemize}
	\item (1) There exists $u \in W^{1, \overrightarrow{p}}(\mathbb{R}^{N})$ such that $\Phi(u)>0$.
	\item (2) $\psi_{i}(u) \geq 0, \,\, \forall i \in \{1,2,...,N\}$ and $J(u)=0 \Leftrightarrow u=0$ . 
	\item (3) There is $r>0$ such that 
	$$
	\sum_{i=1}^{N}\frac{(N-p_{i})}{p_{i}}\int_{\mathbb{R}^{N}}\left|\frac{\partial u}{\partial x_i}\right|^{p_{i}}dx- N\Phi(u)>0 \quad \mbox{for} \quad 0<\|u\|<r.
	$$
	\item (4) By Sobolev embedding, there is a positive $C_1$ such that 
	$$
	\int_{\mathbb{R}^N}|u|^{p_1}\,dx \leq C_1\left( \sum_{i=1}^{N}(\psi_{i}(u))^{\frac{1}{p_i}}\right)^{p^*}, \quad \forall u \in W^{1, \overrightarrow{p}}(\mathbb{R}^{N}) \quad \mbox{with} \quad \Phi(u)>0.
	$$	
	\item (5) If $(u_k) \subset X^r$ is weakly convergent to $u$ in $W^{1, \overrightarrow{p}}(\mathbb{R}^{N})$, by compact embedding of $X^r$ in $L^{q}(\mathbb{R}^N)$ for all $q \in (p_0,p^*)$, we have  
	$$
	\limsup_{k \to +\infty} \Phi(u_{k}) \leq \Phi(u).
	$$
	\item (6) If $(u_k) \subset W^{1, \overrightarrow{p}}(\mathbb{R}^{N})$ is weakly convergent to $u$ in $W^{1, \overrightarrow{p}}(\mathbb{R}^{N})$, then 
	$$
	\liminf_{k \to +\infty}  \left(\sum_{i=1}^{N}\psi_{i}(u_{k})\right) \geq \sum_{i=1}^{N}\psi_{i}(u).
	$$
	\item (7) 	$\displaystyle \inf_{w \in \mathcal{P}}I_2(w)=\inf_{w \in \mathcal{P}^+}I_2(w).$ 	
\end{itemize}

From the above commentaries, the Theorem \ref{ANST} is proved, because all the conditions of Theorem \ref{T1} were verified.

\section{An application involving discontinuous nonlinearity}
In this section we consider the existence of solution for the problem
\begin{equation} \label{DP}
-\Delta u(x) \in \partial G(u(x)), \quad \mbox{a.e. in} \quad \mathbb{R}^N,
\end{equation}
where $N \geq 1$, $G$ is the primitive of a function $g:\mathbb{R} \to \mathbb{R}$ given by $g(s)=f(s)-s$, that is, 
$$
G(s)=\int_{0}^{s}g(t)\,dt=\int_{0}^{s}f(t)\,dt-\frac{1}{2}|s|^{2}=F(s)-\frac{1}{2}|s|^{2}
$$
and $\partial G(s)$ is the generalized gradient of $G$ at $s \in \mathbb{R}$, given by 
$$
\partial G(s)=[\underline{g}(s),\overline{g}(s)]
$$
where
$$
\underline{g}(s)=\displaystyle \lim_{r\downarrow 0}\mbox{ess
	inf}\left\{ g(t);|s-t|<r\right\} \quad \mbox{and} \quad \overline{g}(s)=\lim_{r\downarrow 0}\mbox{ess sup}\left\{
g(t);|s-t|<r\right\} .
$$

When $g$ is a continuous function, which is equivalent to say that $f$ is continuous, we know that $G \in C^{1}(\mathbb{R}, \mathbb{R})$ and in this case  
$$
\partial G(s)=\{g(s)\}, \quad \forall s \in \mathbb{R}.
$$

In this section, we assume that $f$ can have a finite number of  discontinuity points $a_1,a_2, ...., a_p \in \mathbb{R} \setminus \{0\}$. Moreover,  $f$ verifies $(f_1),(f_3)$ and the condition \\ 

\noindent $(f_7)$ \,\, There are $A,B>0$ such that 
$$
	|f(s)| \leq A|s|+B|s|^{q}, \quad \forall s \in \mathbb{R},
$$
for some $q \in (1, 2^{\ast}-1)$ where $2^{\ast}= \frac{2N}{N-2}$ if $N \geq 3$ and $2^*=+\infty$ if $N=1,2$.
	
Since we intend to find a nonnegative solution, in what follows we assume that 
$$
f(s)=0, \quad \forall s<0. 
$$

Hereafter, by a solution we understand as being a function $u \in W_{loc}^{2, \frac{q+1}{q}}(\mathbb{R}^{N}) \cap H^{1}(\mathbb{R}^{N})$ that verifies (\ref{DP}), or equivalently, the problem below
\begin{equation} \label{equivalente}
-\Delta u(x)  +u(x) \in \left[\underline{f}(u(x)), \overline{f}(u(x))\right], \quad \mbox{a.e. in} \quad \mathbb{R}^N.
\end{equation}
For the case where $f$ is a continuous function, the above solution must verify the equation 
$$
-\Delta u(x)  +u(x) = f(u(x)), \quad \mbox{a.e. in} \quad \mathbb{R}^N.
$$

A rich literature is available by now on problems with discontinuous nonlinearities, and we refer the reader to  Alves, Bertone and Gon\c calves \cite{ABG}, Alves and Bertone \cite{AB}, Alves, Gon\c calves and Santos \cite{Abrantes},  Ambrosetti and Turner \cite{turner}, Ambrosetti, Calahorrano and Dobarro \cite{ACD}, Badiale and Tarantelo \cite{Badiale}, Carl, Le and Motreanu \cite{carl}, Clarke \cite{clarke}, Chang \cite{Chang}, Carl and Dietrich \cite{CD1}, Carl and S. Heikkila \cite{CD2,CD3}, Cerami \cite{cerami}, Hu, Kourogenis and Papageorgiou \cite{Hu}, Montreanu and Vargas \cite{MV}, Radulescu \cite{Radulescu} and their references. Several techniques have been developed or applied in their study, such as variational methods for nondifferentiable functionals, lower and upper solutions, global branching, fixed point theorem, and the theory of multivalued mappings.

The main idea is showing the existence of a nontrivial critical point of the functional energy associated with problem $I_3:H^{1}(\mathbb{R}^N) \to \mathbb{R}$ given 
$$
I_3(u)= \frac{1}{2}\int_{\mathbb{R}^{N}}|\nabla u|^{2}dx + \frac{1}{2}\int_{\mathbb{R}^{N}}|u|^{2}dx - \int_{\mathbb{R}^{N}}F(u)dx.
$$
Since $I_3$ is locally Lipschitz, a function $u \in H^{1}(\mathbb{R}^N)$ is a critical point of $I_3$, if $0 \in \partial I_3(u)$, where $\partial I_3(u)$ denotes the generalized gradient of $I_3$. We recall The generalized gradient of $I_3$ at $u$  is the set 
$$
\partial I_3(u)=\{ \mu \in  (H^{1}(\mathbb{R}^N))^*\,:\, I_{3}^{0}(u,v) \geq \langle \mu,v \rangle, \, \forall v \in H^{1}(\mathbb{R}^N) \}
$$
where $I_{3}^{0}(u,v)$ denotes the directional derivative of $I_3$ on $u$ in the direction of $v \in H^{1}(\mathbb{R}^N)$ defined by
$$
I_{3}^{0}(u,v)=\limsup_{h \to 0, \lambda \downarrow 0}\frac{I_3(u+h+\lambda v)-I_3(u+h)}{\lambda}.
$$
The reader can find more details about this subject in Chang \cite{Chang} and Clarke \cite{clarke}.   
 
\vspace{0.5 cm}

The main result this section is the following

\begin{theorem} \label{TD}
	Assume the conditions $(f_1)-(f_3)$ and $(f_7)$. Then, (\ref{DP}) has a nontrivial solution. 
\end{theorem}

As in the previous section, we will show that functional $I_3$ verifies the conditions of Theorem \ref{T1}. In what follows, 
$$
X=H^{1}(\mathbb{R}^N), \quad X^{+}=\{u \in H^{1}(\mathbb{R}^N)\,:\, u(x) \geq 0 \quad \mbox{a.e. in} \quad \mathbb{R}^N\},
$$
$$
X^r=\{u \in H_{rad}^{1}(\mathbb{R}^N) \cap X^+\,:\, 0 \leq u(x) \leq u(y) \quad \mbox{if} \quad 0<|y| \leq |x| \},
$$
$$
J(u) = \psi(u)=\frac{1}{2}\int_{\mathbb{R}^N}|\nabla u|^{2}\,dx
$$
and
$$
\Phi(u) = \int_{\mathbb{R}^{N}}G(u)\,dx.
$$
By using these notations, 
$$
I_3(u)=J(u)-\Phi(u), \quad \forall u \in X=H^{1}(\mathbb{R}^N).
$$ 
A well known argument shows that $\Phi$ is locally Lipschitz, $J \in C^{1}(H^{1}(\mathbb{R}^N), \mathbb{R})$ and $I_3$ is locally Lipschitz.

Moreover, as in the previous sections, we set $\ast(t,u):=u_{t}: \mathbb{R}^N \to \mathbb{R}$  by 
$$
u_{t}(x)=
\left\{
\begin{array}{l}
u \left(\frac{x}{t}\right), \quad \mbox{ for } \quad t \not= 0, \\
0, \quad \mbox{for} \quad t=0.
\end{array}
\right.
$$	
By definition of $\psi$ and $\Phi$, it follows that 
$$
\Phi(u_{t})=t^{N}\Phi(u) \mbox{ and } \psi(u_{t})=t^{N-2}\psi(u), \quad \forall t \geq 0 \quad \mbox{and} \quad  \forall u \in H^{1}(\mathbb{R}^N).
$$
From the above commentaries, we have proved $(X_1)-(X_4)$ and $(F_2)$ of Theorem \ref{T1}. Moreover, as in the previous sections, for each $u \in H^{s_{n}}(\mathbb{R}^{N})$, the application $t \longmapsto u_{t}$ is continuous, and so, $(X_5)$ is also proved.  

Next, as in Section 4, we will prove the other conditions of Theorem \ref{T1} also occur.  Setting 
	$$
	\begin{array}{cccl}
	Q:& H^{1}(\mathbb{R}^{N})& \longrightarrow& X^r \\
	&u& \longmapsto &(u^{+})^{\ast}	,
	\end{array}
	$$
	where  $(u^{+})^{\ast}$ is the Schwartz's symmetrization of $u^{+}$, standard arguments show that $(X_6)-(X_8)$ also hold. 
	\begin{claim}( {\bf Proof of $(F_1)$})
		There is $u \in H^{1}(\mathbb{R}^{N})$ such that $\Phi(u)>0$ and $\Phi(0)=0$. 
	\end{claim} 
	\begin{proof} The claim follows with the same arguments explored in Claim \ref{4302}. 
	\end{proof}	

	\begin{claim} ( {\bf Proof of $(F_3)$})
		There exists $r>0$ such that  
		$$
		(N-2)\psi(u)>N\Phi(u), \quad \mbox{for} \quad 0<||u||< r. 
		$$ 
	\end{claim}
	\begin{proof} See proof of Claim \ref{4303}
	\end{proof}
	
	\begin{claim} ( {\bf Proof of $(F_4)$})
			Let $(u_{k})$ be a sequence in $H^{1}(\mathbb{R}^{N})$ with $\Phi(u_{k})\geq 0$ for all $k \in \mathbb{N}$. If $(J(u_{k}))$ is bounded,  we have that  $(u_{k})$ is also bounded. Moreover, if $J(u_{k})\rightarrow 0$, then $||u_{k}|| \rightarrow 0$. 
		\end{claim}
		\begin{proof} An immediate consequence of $J$. 
	\end{proof}
	
	\begin{claim} ( {\bf Proof of $(F_5)$})
		If  $(u_{k})$ is weakly convergent to $u$ in $H^{1}_{rad}(\mathbb{R}^{N})$, then 
		$$
		\limsup_{k \rightarrow \infty}\Phi(u_{k}) \leq \Phi(u).
		$$
	\end{claim}
	
	\begin{proof}
	See	proof of Claim \ref{a4305}.
	\end{proof}
	
	\begin{claim} ( {\bf Proof of $(F_6)$})
		Se $(u_{k})$ is weakly convergent to $u$ in $H^{1}(\mathbb{R^{N}})$, then   
		$$ 
		\psi(u) \leq \liminf_{k \rightarrow +\infty}\psi(u_{k}).
		$$
	\end{claim}	
	\begin{proof}
		See proof of Claim \ref{a4306}.
	\end{proof}
		\begin{claim}  The functional $I_3$ verifies the equality below  
		$$
		\inf_{w \in \mathcal{P}}I_3(w)=\inf_{w \in \mathcal{P}^+}I_3(w).
		$$ 	
		\end{claim}
	\begin{proof}
		See proof of Claim \ref{igualdade}
	\end{proof}
 From the above analysis, we can deduce that the functional $I_3$ verifies the conditions of Theorem \ref{T1}.  From this, there is $u_0 \in H^{1}(\mathbb{R}^N)$ such that  
$$
0 \in \partial I_3(u_{0}) \quad \mbox{and} \quad I_3(u_0)= \inf_{w \in \mathcal{P}}I_3(w)>0.
$$
As $J$ is $C^{1}(H^1(\mathbb{R}^N), \mathbb{R})$, we have 
$$
J'(u_0)  \in \partial \Phi(u_0).
$$
On the other hand, as 
$$
\Phi(u)=\Phi_1(u)+\Phi_2(u), \quad \forall u \in H^{1}(\mathbb{R}^N)
$$
with
$$
\Phi_1(u)=\int_{\mathbb{R}^N}F(u)\,dx \quad \mbox{and} \quad \Phi_2(u)=\frac{1}{2}\int_{\mathbb{R}^N}|u|^2\,dx,
$$
it follows that
\begin{equation} \label{final}
J'(u_0)+\Phi'_2(u_0) \in \partial \Phi_1(u_0).
\end{equation}
Here, we have used that $\Phi_2 \in C^{1}(H^1(\mathbb{R}^N), \mathbb{R})$.

By using well known result found in \cite{rbgiocl}, we know that 
$$
\partial \Phi_1(u_0) \subset [\underline{f}(u_0),\overline{f}(u_0)].
$$
This combined with (\ref{final}) guarantee the existence of a mensurable function $\rho:\mathbb{R}^N \to \mathbb{R}$ verifying 
\begin{equation} \label{rho}
\rho (x) \in [\underline{f}(u_0(x)),\overline{f}(u_0(x))], \quad \mbox{a.e. in} \quad \mathbb{R}^N, 
\end{equation}
such that $u_{0}$ is a weak solution of the problem  
$$
	-\Delta u_{0}+u_{0}=\rho, \quad \mbox{in} \quad \mathbb{R}^N.
$$ 
By (\ref{rho}), $\rho \in L_{loc}^{\frac{p+1}{p}}(\mathbb{R}^N)$, then the elliptic regularity gives  $u_{0}\in W_{loc}^{2,\frac{q+1}{q}}(\mathbb{R}^{N}) \cap H^{1}(\mathbb{R}^{N})$, 
showing the $u_{0}$ is a nontrivial solution of (\ref{equivalente}). 

\vspace{0.5 cm}

\noindent {\bf Conflict of Interest:} The authors declare that they have no conflict of interest.

\end{document}